\documentclass[12pt, reqno]{amsart}

\usepackage{amsmath}
\usepackage{amsthm}
\usepackage{amsfonts}
\usepackage{amssymb}
\usepackage{verbatim}
\usepackage{graphicx}
\usepackage{mathtools}
\usepackage[margin=1.0in]{geometry}
\usepackage{url}
\usepackage{enumerate}
\usepackage[pdftex,bookmarks=true]{hyperref}
\usepackage[normalem]{ulem}
\usepackage[usenames, dvipsnames]{color}

\theoremstyle{plain}
\newtheorem{theorem}{Theorem}[section]
\newtheorem{conjecture}[theorem]{Conjecture}

\newtheorem{lemma}[theorem]{Lemma}
\newtheorem{corollary}[theorem]{Corollary}

\newtheorem{question}[theorem]{Question}

\theoremstyle{definition}

\newcommand\C{{\mathbb{C}}}
\newcommand\N{{\mathbb{N}}}

\newcommand\R{{\mathbb{R}}}
\renewcommand\P{{\mathbb{P}}}
\newcommand\E{{\mathbb{E}}}

\newcommand{\ep}{\varepsilon}

\newcommand{\ba}{\[\begin{aligned}}
	\newcommand{\ea}{\end{aligned}\]}

\let\oldtocsection=\tocsection

\let\oldtocsubsection=\tocsubsection

\let\oldtocsubsubsection=\tocsubsubsection

\renewcommand{\tocsection}[2]{\hspace{0em}\oldtocsection{#1}{#2}}
\renewcommand{\tocsubsection}[2]{\hspace{1em}\oldtocsubsection{#1}{#2}}
\renewcommand{\tocsubsubsection}[2]{\hspace{2em}\oldtocsubsubsection{#1}{#2}}

\title{On the First Non-Universal Term in Random Polynomial Real Zeros}
\author{Phuc Lam and Oanh Nguyen}
\address{
Department of Applied Mathematics \\
Brown University \\
Providence, RI 02906 \\
USA \\}
\email{phuc\_lam@brown.edu}
\address{
Department of Applied Mathematics \\
Brown University \\
Providence, RI 02906 \\
USA \\}
\email{oanh\_nguyen1@brown.edu}
\thanks{}
\subjclass[2010]{}
\date{}

\begin{document}
 
\maketitle

\begin{abstract}
    Let $P_n(x) = \sum_{k=0}^{n} \xi_k x^k$ be a Kac random polynomial, where the coefficients $\xi_k$ are i.i.d.\ copies of a given random variable $\xi$. Based on numerical experiments, it has been conjectured that if $\xi$ has mean zero, unit variance, and a finite $(2+\varepsilon_0)$--moment for some $\varepsilon_0>0$, then
    \[
    \mathbb{E}[N_{\mathbb{R}}(P_n)] \;=\; \frac{2}{\pi} \log n + C_{\xi} + o_n(1),
    \]
    where $N_{\mathbb{R}}(P_n)$ denotes the number of real roots of $P_n$, and $C_{\xi}$ is an absolute constant depending only on $\xi$, which is nonuniversal. Prior to this work, the existence of $C_{\xi}$ had only been established by Do--Nguyen--Vu (2015, \emph{Proc.\ Lond.\ Math.\ Soc.}) under the additional assumption that $\xi$ either admits a $(1+p)$--integrable density or is uniformly distributed on $\{\pm 1, \pm 2, \dots, \pm N\}$. In this paper, using a different method, we remove these extra conditions on $\xi$, and extend the result to the setting where the $\xi_k$ are independent but not necessarily identically distributed. Moreover, this proof strategy provides an alternative description of the constant $C_{\xi}$, 
    and this new perspective serves as the key ingredient in establishing that $C_{\xi}$ 
    depends continuously on the distribution of $\xi$.
    \end{abstract}

\section{Introduction}
Let $\xi$ be a random variable with zero mean and unit variance. Consider the Kac polynomial
\begin{equation*}
    P_n(x) := \sum_{k = 0}^{n} \xi_k x^k,
\end{equation*}
where $\xi_k$ are i.i.d. copies of $\xi$. For any $D \subseteq \C$, let $N_{P_n}(D)$ be the number of roots of $P_n$ (counting multiplicity) in $D$. 

The Kac polynomial is a classical model of random polynomials and has received renewed interest in the last two decades; see, for example, the very incomplete list \cite{BharuchaSambandham, Farahmand1996, SheppVanderbei1995, ibragimov1997roots, shiffman2003equilibrium, TaoVu2010, kabluchko2013roots, ORourke2019pairing, michelen2021real, OanhNguyenVu2022, nguyen2024hole, michelen2024central, michelen2025limit} and the references therein. The study of real zeros of Kac polynomials dates back to Bloch and Polya \cite{BP} who studied the case when $\xi$ is uniformly distributed in $\{-1,0,1\}$ and established the upper bound
$\mathbf{E} N_n = O(n^{1/2}).$
In a series of breakthrough papers \cite{LO1,LO2,LO3}
in the early 1940s, Littlewood and Offord proved (for many atom variables
$\xi$ such as Gaussian, Rademacher (uniform on $\{-1, 1\}$), or uniform on $[-1,1]$) that the number of real roots is bounded above and below by a polylogarithmic function of $n$ with high probability.

The first precise asymptotic was obtained in Kac’s seminal papers \cite{Kac1943average}, which introduced the now well-known Kac–Rice formula
\begin{equation}  
	\E N_n = \int_{-\infty}^{\infty} dt \int_{-\infty}^{\infty} |y| \, p(t,0,y) \, dy,\notag
\end{equation}
where $p(t,x,y)$ is the joint probability density for $(P_n(t), P_n'(t)) = (x, y)$. For Gaussian coefficients, this formula can be evaluated explicitly to yield
\begin{equation}  
	\E N_n = \frac{1}{\pi} \int_{-\infty}^{\infty}
	\sqrt{ \frac{1}{(t^2 - 1)^2} + \frac{(n+1)^2 t^{2n}}{(t^{2n+2} - 1)^2} } \, dt
	= \left( \frac{2}{\pi} + o(1) \right) \log n.\notag
\end{equation} 
 
Evaluating the Kac-Rice formula for non-Gaussian coefficients is notoriously difficult as it involves integrating the joint density of $P_n$ and $P_n'$. For discrete distributions, Erd\H{o}s and Offord~\cite{EO} in 1956 developed a completely different method -- combining combinatorial arguments with a connection between real roots and sign changes -- to handle the Rademacher case in which $\xi_i$ takes values $\pm 1$ with equal probability. They proved that with probability $1 - o\!\left(\frac{1}{\sqrt{\log\log n}}\right)$,
\begin{equation}  
	N_{n,\xi} = \frac{2}{\pi} \log n + o\!\left( \log^{2/3} n \, \log\log n \right).\notag
\end{equation}

In the late 1960s and early 1970s, Ibragimov and Maslova \cite{Ibragimov1968average, Ibragimov1971average} extended this method and proved that $\E N_{P_n}(\R)=(1+o(1))\frac{2}{\pi}\log n$ is universal if $\xi$ has mean zero and belongs to the domain of attraction of the normal law.

In the last two decades, the method of universality was developed in random matrix theory and subsequently adopted for random polynomials with the pioneering work by Tao and Vu \cite{TVpoly}. Developing this method, Do, Vu, and the second author \cite{DoONguyenVu2018} showed that for any $\xi$ with mean zero, variance one, and a finite $(2+\varepsilon_0)$--moment for some $\varepsilon_0>0$, the error term is indeed of order $O(1)$:$$\E N_{P_n}(\R)=\frac{2}{\pi}\log n+O(1).$$

While the leading-order behavior of the expected number of real roots is now known to be universal, it is natural to ask whether the finer structure continues to reflect the law of the coefficients. 
Numerical experiments such as those shown in Figure \ref{fig1} below suggest the following conjecture.
 \begin{conjecture}\label{conj} If the random variable $\xi$ has mean 0, variance 1, and a bounded $(2+\ep_0)$-moment for some $\ep_0>0$, then there exists a constant $C_{\xi}$ depending only on $\xi$ such that
 	\begin{equation}
 		\E N_{P_n}(\R) = \dfrac{2}{\pi}\log n + C_{\xi} + o(1) \text{ as $n\to\infty$.  }\notag
 	\end{equation}
Moreover, the constant $C_{\xi}$ is non-universal. For instance, $C_{\xi}$ may be different from $C_{\mathrm{Gau}}$ if $\xi$ is non-Gaussian.
\end{conjecture}
\begin{figure}[h!] \label{fig1}
	\includegraphics[width=.7\textwidth]{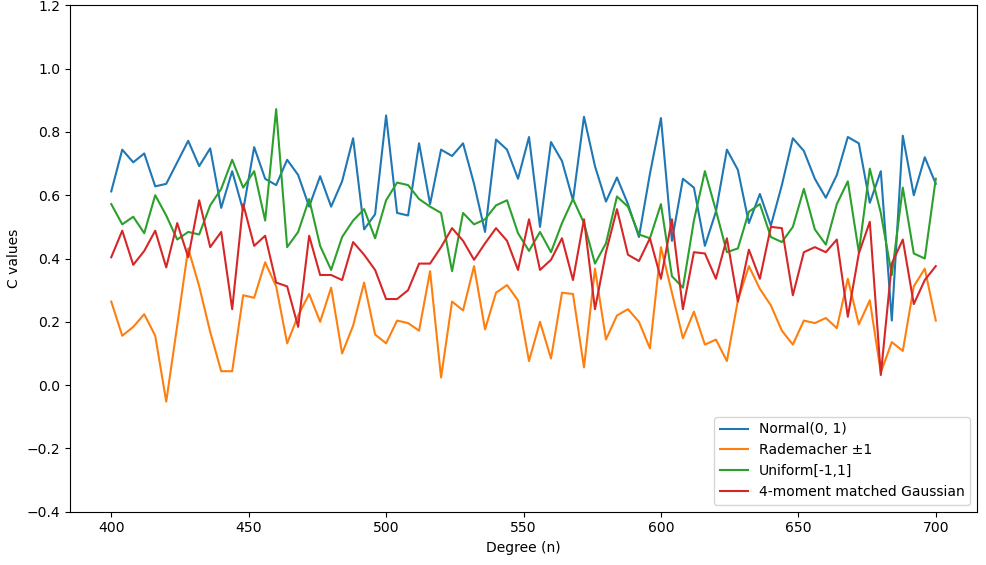} 
	\caption{Sampled value of $C_{\xi}$ for 4 different coefficients distributions: standard Gaussian, uniform in $[-1, 1]$, Rademacher, and a distribution that matched first four moments with the standard Gaussian: $\mathbb{P}(X=\pm \frac{1}{\sqrt{2}})=\frac{4}{9}\text{ each}, 
		\mathbb{P}(X=\pm \sqrt{5})=\frac{1}{18} \text{ each}$.}
\end{figure}

In the Gaussian case, the existence of this constant was first established by Wilkins~\cite{wilkins1988asymptotic} and later by Edelman and Kostlan~\cite{EdelmanKostlan} through a careful evaluation of the Kac–Rice formula.
\begin{equation} \label{eq:GaussianConstant}
	\E N_{n, N(0,1)} = \frac{2}{\pi} \log n + C_{\mathrm{Gau}} + o(1).
\end{equation}
Here $C_{\mathrm{Gau}} \approx 0.625738072\ldots$ is given by a complicated but explicit integral. Moreover, the $o(1)$ term can be expanded as a convergent series of explicit functions of $n$, yielding a complete Taylor expansion.

 Do--Nguyen--Vu \cite{DoHoiNguyenVu} made substantial progress toward this conjecture and confirmed the existence of $C_{\xi}$ with the extra assumption that $\xi$ is either uniform on $\{\pm 1, \pm 2, \ldots, \pm N\}$ or continuous with a $p$-integrable density function for some $p>1$.

 \vspace{5mm}
 
In the first main result of this paper, we remove these extra conditions.
\begin{theorem} [Existence of $C_{\xi}$ in i.i.d. setting]\label{thm:mainthm1}
Let $\ep_0$ and $M_0$ be any positive constants. Consider the sequence of Kac polynomials $(P_n)_{n \ge 1}$, where $\xi_1, \xi_2, \dots$ are i.i.d. copies of a random variable $\xi$ with zero mean, unit variance, and $\E |\xi|^{ 2 + \ep_0}\le M_0$. Then there exist constants $C_{\xi, I}, C_{\xi}$ such that as $n\to\infty$,
\begin{equation}
    \E N_{P_n}(I) = \dfrac{1}{2\pi}\log n + C_{\xi, I} + o(1), \label{eq:I1}
\end{equation}
where $I$ is any of the intervals $(-\infty, -1], [-1, 0], [0, 1], [1, \infty)$, with endpoints included or excluded. Summing over these intervals yields
\begin{equation}
	\E N_{P_n}(\R) = \dfrac{2}{\pi}\log n + C_{\xi} + o(1).\notag
\end{equation}
\end{theorem}
This theorem is a special case of the following result, where the assumption that the random variables $\xi_i$ are identically distributed is dropped.
 \begin{theorem}[Existence of $C_{\xi}$ in non-i.i.d. setting]\label{thm:mainthm}
 	Let $\ep_0$ and $M_0$ be any positive constants. Consider the sequence of Kac polynomials $(P_n)_{n \ge 1}$, where $\xi_1, \xi_2, \dots$ are independent random variables with zero mean, unit variance, and bounded $(2+\ep_0)$-moment: $\E |\xi_i|^{ 2 + \ep_0}\le M_0$ for all $i$. Assume furthermore that there exists a random variable $ \xi$ and a positive constant $c$ such that for all $\ep\in (0, c)$ and all $i\in \N$,
 	\begin{equation}\label{cond:decay0}
 		\P(|\xi_i|\le \ep, \xi_i\neq 0) \le \frac{1}{c}\P(| \xi|\le \ep,  \xi\neq 0).
 	\end{equation}

 	Then there exists a constant $C_{\xi_0, \xi_1, \dots, (0, 1]}$ such that as $n\to\infty$,
 	\begin{equation}
 		\E N_{P_n}((0, 1] ) = \dfrac{1}{2\pi}\log n + C_{\xi_0, \xi_1, \dots, (0, 1]} + o(1). \label{eq:I}
 	\end{equation}
 \end{theorem}

The next corollary provides an alternative description of the constants $C_{\xi, I}$.
 \begin{corollary}\label{cor:C}
 	Under the assumptions of Theorem \ref{thm:mainthm}, we have
 	\begin{equation}
 		C_{\xi_0, \xi_1, \dots, (0, 1]} =C_{\mathrm{Gau}, (0, 1]}- \frac{1}{2\pi}\log 2+\lim_{C\to\infty}\left [ \E N_{P_{\infty}} \left (\left (0, 1-\frac1C\right )\right ) -\frac{1}{2\pi}\log C\right ], \label{eq:C}
 	\end{equation}
 where the convergence in $C$ is uniform over all distributions of the $\xi_i$, given $\ep_0$ and $M_0$.
 \end{corollary}

This corrolary allows us to show that the constant $C_{\xi}$ in Theorem \ref{thm:mainthm1} is in fact continuous.  
 \begin{theorem}[Continuity of $C_{\xi}$]\label{thm:cont}
 	Let $\mu_{n}$ and $\mu_{\infty}$ be probability measures on $\R$ with zero mean and unit variance satisfying the following properties.
 	\begin{enumerate}
 		\item The measures $\mu_{n}$ converge to $\mu$ in distribution as $n\to\infty$. Moreover, $\lim_{n\to\infty}\mu_{n}(0)= \mu_\infty(0)<1$.
 		\item There exist $M_0 <\infty$ and $\ep_0>0$ such that
 		$\E_{\xi\sim \mu_n}|\xi|^{2+\ep_0} \le M_0 $ for all $n\in \N\cup\{\infty\}$.
 	\end{enumerate}
 Let $C_{\mu_{n}}$ be the constant in Theorem \ref{thm:mainthm1} when $\xi_0, \xi_1, \dots$ are i.i.d. $\mu_{n}$.	Then $C_{\mu_{n}}\to C_{\mu_\infty}$ as $n \rightarrow \infty$.
 \end{theorem}

 Towards the second part of Conjecture \ref{conj}, if one allows $\xi$ to have mass at 0, then the constant $C_{\xi}$ can be made as large as possible. To see this, let $p=\P(\xi=0)$, then the multiplicitiy at $0$ has expectation
 \begin{eqnarray}
 	\E N_{P_n}(\{0\}) = \sum_{m=1}^{n} m \P(\xi_0=\dots =\xi_{m-1}=0) \P(\xi_{m}\neq 0)\to \sum_{m=1}^{\infty} m p^{m} (1-p) =\frac{p}{1-p}\label{eq:0}
 \end{eqnarray}
 which can take any value in $[0, \infty)$. This, in combination with \eqref{eq:uni} that we prove later, implies that the constant $C_{\xi}$ can be arbitrarily large.
 \begin{lemma}\label{lm:nonUniversality}
 	For all $M > 0$, there exists a distribution $\xi$ with zero mean, unit variance, and finite $(2 + \ep_0)$-moment such that the constant terms $C_{\xi, [0, 1]}, C_{\xi, [-1, 0]}$ and $C_{\xi}$ in Theorem \ref{thm:mainthm1} are at least $M$.
 \end{lemma}
For random variables with no mass at 0, tackling the second part of Conjecture \ref{conj} perhaps involves answering a more challenging question.
 \begin{question}
 	Can the constant $C_{\xi}$ be explicitly calculated from the moments of the random variable $\xi$?
 \end{question}
 We reiterate the following question from \cite{DoHoiNguyenVu}: even obtaining a practical bound on $C_{\pm 1}$, the constant corresponding to the Rademacher distribution, is already difficult—for instance, proving that $|C_{\pm 1}|\le 10$.

{\bf Organization.} In Section \ref{sec:sketch}, we provide a sketch of the proof of Theorems \ref{thm:mainthm1} and \ref{thm:mainthm}. The full proof will be carried out in Sections \ref{sec:uni}, \ref{sec:ConvFarFrom01}, and \ref{sec:ConvNear0}. 
Proof of Corollary \ref{cor:C} is presented in Section \ref{sec:cor}.
The proof of Theorem \ref{thm:cont} is in Section \ref{sec:cont}. Proof of Lemma \ref{lm:nonUniversality} can be found in Section \ref{prof:nonuni}.
\section{Proof skeleton of Theorems \ref{thm:mainthm1} and \ref{thm:mainthm}}\label{sec:sketch}
 Our method, after reducing to the non-universal interval $[0, 1-1/C]$, is completely different from that of \cite{DoHoiNguyenVu}. 
 
Firstly, we note that it suffices to show Equation \eqref{eq:I} for $I=[0, 1]$, namely
 \begin{equation}
 	\E N_{P_n}((0, 1]) = \dfrac{1}{2\pi}\log n + C_{\xi, (0, 1]} + o(1), \label{eq:01}
 \end{equation}
 
  Indeed, the number of roots of $P_n(x)$ in $[-1, 0)$ is the same as the number of roots in $(0, 1]$ of 
 $$P_n(-x) =\sum_{i=0}^{n} ((-1)^{i}\xi_i) x^i$$
 which is of the same form as $P_n$ and satisfies the same properties as $P_n$ (except that the random variables  $ ((-1)^{i}\xi_i) $ may not be identically distributed, in view of Theorem \ref{thm:mainthm1}. However, the proof would go similarly). Therefore, \eqref{eq:01} implies the same equation for $I = [-1, 0)$.
 
Moreover, Equation \eqref{eq:01} implies the same asymptotics for the interval $[0, 1]$, up to an additive constant thanks to \eqref{eq:0} which asserts that the expected multiplicity at $0$ converges (for non-i.i.d. case, to prove convergence of the series in \eqref{eq:0}, one can use the fact that the probability that the $\xi_i$ is 0 is bounded away from 1, see for instance \eqref{eq:small}). 

The number of roots of $P_n(x)$ in the interval $[1, \infty)$ is the same as the number of roots of $\hat P_n(x)$ in the interval $(0, 1]$ where
 \begin{equation}
 \hat P_n(x) = x^{n} P_n(1/x) = \sum_{i=0}^{n} \xi_{n-i} x^{i}.\notag
 \end{equation}
 Since the sequence $(\xi_{n-i})_{i=0}^{n}$ has the same distribution as  $(\xi_{i})_{i=0}^{n}$ under the i.i.d.  setting and satisfies the same properties as $(\xi_{i})_{i=0}^{n}$ under the non-i.i.d.  setting, Equation \eqref{eq:01} also holds for $\hat P_n$, and hence the same asymptotics applies to $N_{P_n}([1, \infty))$. Consequently, the result also extends to $N_{P_n}((-\infty, -1])$ by the transformation $x\mapsto -x$ as before.

 \vspace{5mm}
 
It follows that, for the rest of the paper, it suffices to prove \eqref{eq:01}. To this end, we further decompose the interval into
$$(0, 1]= [1-\delta, 1] \cup (1-\delta', 1-\delta) \cup (0, 1-\delta'] $$
where $\delta, \delta'$ are positive small constants. For any $\ep>0$, we need to show that there exists a constant $C$, depending only on $\ep_0, M_0$ and $\ep$, a constant $\delta$ that may depend additionally on the distribution of $\xi$, and for sufficiently large $n$, the following hold:
\begin{equation}
	\left |	 \E N_{P_n}([1-C^{-1}, 1]) - \left (\frac{1}{2\pi}\log\frac{n}{C}+ C_{\mathrm{Gau}, [0, 1]}- \frac{1}{2\pi}\log 2\right ) \right |\le \ep; \label{eq:uni}
\end{equation}
\begin{equation}
	\E N_{P_n}((1-\delta, 1-C^{-1})) \text{ converges as $n\to\infty$};\label{eq:far0}
\end{equation}
\begin{equation}
	\E N_{P_n}((0, 1-\delta])   \le \ep.\label{eq:near0}
\end{equation}

\vspace{5mm}
The proof of \eqref{eq:uni} is given in Section~\ref{sec:uni}, where we use universality to reduce to the Gaussian case. To prove \eqref{eq:far0}, writing $[\delta, 1 - C^{-1}] = [r, R]$, we first establish the weak convergence of $(P_n(z))_{z\in [r, R]}$ to $(P_\infty(z))_{z\in [r,R]}$ as random analytic functions. 
It follows that the number of real zeros of $P_n$ in $[r,R]$ converges in distribution to that of $P_\infty$, provided that $P_\infty$ has neither multiple roots nor zeros at the endpoints.
Finally, uniform integrability of these zero counts yields convergence in mean. 
The details are given in Section~\ref{sec:ConvFarFrom01}, where we emphasize the importance of showing that the infinite series $\sum_{k\ge 0} \xi_k x^k$ has no double roots. 
In Section~\ref{sec:shortTaylor}, we present a short proof of this fact in the case where $\xi$ has a symmetric distribution on $\{0,\pm 1\}$ with positive mass at~0. A short proof of \eqref{eq:near0} is presented in Section \ref{sec:ConvNear0}.

 \section{Proof of \eqref{eq:uni}: Real roots near 1}\label{sec:uni}

Note that Theorem \ref{thm:mainthm} still holds when all coefficients are i.i.d. $\sim \mathcal N(0, 1)$. When $x$ is close to $1$, \cite{nguyenvuCLT} (see also \cite{NNV16} for the iid case) established universality, as follows.

\begin{theorem}\cite[Corollary 2.2]{nguyenvuCLT}\label{thm:universality}
Under the assumption of Theorem \ref{thm:mainthm}, for all $\ep>0$, there exists a constant $C_0$, depending only on $\ep, \ep_0$ and $M_0$ such that for all $n$ sufficiently large (dependending only on $\ep, \ep_0$ and $M_0$), it holds that 
\begin{equation*} 
   \left |\E N_{P_n}([1 - 1/C_0, 1]) - \E N_{P_n^g}([1 - 1/C_0, 1])\right| \le \ep.
\end{equation*}
\end{theorem}

 Next, we evaluate the mean number of real roots for the Gaussian setting. By the Kac-Rice formula, it is well-known (see example, \cite{EdelmanKostlan}) that
\begin{eqnarray*}
	\E N_{P_n^g}([1 - 1/C_0, 1])&=& \E N_{P_n^g}([0, 1]) - \E N_{P_n^g}([0, 1-1/C_0)\\ 
&=&  \frac{1}{2\pi}\log n+ C_{\mathrm{Gau}, [0, 1]} + o(1)	- \int_{0}^{1-C_0^{-1}}
	\frac{1}{\pi}
	\sqrt{ \frac{1}{(t^{2}-1)^{2}} 
		- \frac{(n+1)^{2} t^{2n}}{(t^{2n+2}-1)^{2}} } \, dt.
\end{eqnarray*}
 For $t\in [0, 1-1/C_0)$, it is not hard to see that
 $$\sqrt{ \frac{1}{(t^{2}-1)^{2}} 
 	- \frac{(n+1)^{2} t^{2n}}{(t^{2n+2}-1)^{2}} } = \frac{1}{1-t^{2}} + o(1) $$
 which makes
\begin{eqnarray*}
	\frac1\pi\int_{0}^{1-C_0^{-1}} 
	\sqrt{ \frac{1}{(t^{2}-1)^{2}} 
		- \frac{(n+1)^{2} t^{2n}}{(t^{2n+2}-1)^{2}} } \, dt &=&  \frac{1}{2\pi}\log\frac{1+t}{1-t}\bigg\vert _{0}^{1-C_0^{-1}} + o(1)\\
		&=&  \frac{1}{2\pi}\log(2-C_0^{-1})+\frac{1}{2\pi} \log C_0+ o(1).
\end{eqnarray*}
Therefore, 
\begin{eqnarray*}
	\E N_{P_n^g}([1 - 1/C_0, 1]) &=&  \frac{1}{2\pi}\log n-\frac{1}{2\pi}\log C_0+ C_{\mathrm{Gau}, [0, 1]}- \frac{1}{2\pi}\log 2  + R_{C_0}.
\end{eqnarray*}
where $|R_{C_0}|\le \ep$ for sufficiently large  $n$.  

\section{Proof of \eqref{eq:far0}: Roots away from 0 and 1}\label{sec:ConvFarFrom01}
Let $P_{\infty}$ be the following random Taylor series 
$$P_{\infty}(z) := \sum_{k = 0}^{\infty} \xi_k z^k.$$
The main objective of this section is to prove \eqref{eq:far0} which we restate in the following lemma. 
\begin{lemma}\label{lem:ConvFarFrom01}
    Fix any fixed $r<R$ in $ (0, 1)$, we have
    $$\lim_{n \rightarrow \infty} \E N_{P_n}([r, R]) = \E N_{P_{\infty}}([r, R]) < \infty.$$
\end{lemma}
To prove the limit, we follow the general strategy in a paper by Iksanov--Kabluchko--Marynych \cite{Iksanov2016}. We first show that  $(P_n(z))_{z \in [r, R]} $ converges weakly  to $(P_{\infty}(z))_{z \in [r, R]}$ as random analytic functions. A standard argument applies to yield that the number of real zeros in $[r, R]$ of $P_n$ converges weakly to that of $P_\infty$, assuming that $P_{\infty}$ does not have double roots. Showing that the probability of $P_{\infty}$ having no double roots is zero is a challenging problem, which was only recently resolved in a beautiful paper by Michelen--Yakir~\cite{MichelenYakir2025}. (In Section~\ref{sec:shortTaylor}, we present an elegant alternative proof in the case where $\xi$ is a random variable taking values in $\{-1,0,1\}$ with probabilities $\tfrac{1-q}{2}, q, \tfrac{1-q}{2}$, respectively, for some $q>0$. For example, this includes the uniform distribution on $\{-1,0,1\}$.We note that even this case was not covered in the previous work by Do--Nguyen--Vu \cite{DoHoiNguyenVu}.) Finally, we prove uniform integrability of these numbers of real roots to get convergence in mean. These steps will be done, respectively, in Sections \ref{sec:functional}, \ref{sec:distribution}, \ref{sec:mean}.
The fact that the limit is bounded is a direct corrolary of Lemma \ref{lem:momentsOfNumRoots} presented in Section \ref{sec:mean}.

\subsection{Functional limit theorem}\label{sec:functional} For any $0 < R < 1$, consider the half-disk
\begin{equation*}
    D_{R} := \{z \in \mathbb{C}: |z| \le R\}.
\end{equation*}
Let $\mathcal{H}(D_{R})$ be the Banach space of complex-valued function continuous on $D_{R}$ and analytic in the interior of $D_{R}$, endowed with the supremum norm. Let 
\begin{equation*}
    \mathcal{H}_{\mathbb{R}}(D_{R}) := \{f \in \mathcal{H}(D_R): f(\mathbb{R} \cap D_{R}) \subseteq \mathbb{R} \},
\end{equation*}
i.e. the closed subset of $\mathcal{H}(D_{R})$ taking real values on $\mathbb{R} \cap D_{R}$. We first show the following functional limit theorem.

\begin{lemma}\label{lem:FunctionalLimitTheorem}
    Weakly on the space $\mathcal{H}_{\mathbb{R}}(D_{R})$, we have, as $n \rightarrow \infty$,
    \begin{equation*}
        (P_n(z))_{z \in D_{R}} \xrightarrow[]{w} (P_{\infty}(z))_{z \in D_{R}}.
    \end{equation*}
    Here, $\xrightarrow[]{w} $  denotes weak convergence of random elements taking values in a metric space.
\end{lemma}

\begin{proof} 
By Prokhorov’s theorem, it suffices to verify convergence in finite dimensional and tightness.
	
	\textbf{Convergence of finite-dimensional distributions. }
    We show that for all $d \in \mathbb{N}$ and $z_1, \dots, z_d \in D_{R}$, as $n \rightarrow \infty$.
    \begin{equation*}
        \begin{pmatrix}
            \text{Re } P_n(z_1) \\
            \text{Im } P_n(z_1) \\
            \vdots \\
            \text{Re } P_n(z_d) \\
            \text{Im } P_n(z_d)
        \end{pmatrix} \xrightarrow[]{d} \begin{pmatrix}
            \text{Re } P_{\infty}(z_1) \\
            \text{Im } P_{\infty}(z_1) \\
            \vdots \\
            \text{Re } P_{\infty}(z_d) \\
            \text{Im } P_{\infty}(z_d)
        \end{pmatrix}.
    \end{equation*}

    Even better, we will show that the convergence above is in $L^2$:
    \begin{align}
        &\ \sum_{k=1}^d \left((\text{Re } P_n(z_k) - \text{Re } P_{\infty}(z_k))^2 + (\text{Im } P_n(z_k) - \text{Im } P_{\infty}(z_k))^2\right) \nonumber \\
        = & \ \sum_{k=1}^d  \left| \sum_{j > n} \xi_j z_k^j \right|^2 \le \sum_{k = 1}^d \left( \sum_{j > n} |\xi_j| \ |z_k|^j \right)^2 \le \sum_{k = 1}^d \left( \sum_{j > n} |\xi_j| \ R^j \right)^2 = d \left( \sum_{j > n} |\xi_j| \ R^j \right)^2 \nonumber \\
        = & \ d\sum_{j \ge 2(n+1)} \left( \sum_{u = n+1}^{j - (n+1)} |\xi_u| \ |\xi_{j-u}| \right) R^j, \label{eq:BoundForFinDimConv}
    \end{align}
    Note that
    \begin{equation*}
        \E |\xi_u| \ |\xi_{j-u}| \le (\E \xi_u^2)^{1/2} (\E \xi_{j-u}^2)^{1/2} = 1,
    \end{equation*}
    so taking the expectation of both sides of \eqref{eq:BoundForFinDimConv} gives
    \begin{align*}
        &\ \E\left(\sum_{k=1}^d \left((\text{Re} P_n(z_k) - \text{Re} P_{\infty}(z_k))^2 + (\text{Im} P_n(z_k) - \text{Im} P_{\infty}(z_k))^2\right)\right) \\
        \le &\ d \sum_{j \ge 2(n+1)} (j - 2(n+1))R^j = d\left( \sum_{j \ge 0} j R^j\right) R^{2(n+1)} \longrightarrow 0,
    \end{align*}
    as $n \rightarrow \infty$. This establishes the desired $L^2$ convergence.

 \vspace{5mm}
 \textbf{Tightness.}   It remains to show that $(P_n(z))_{z \in D_{R}}$ is a tight sequence on $\mathcal{H}_{\mathbb{R}}(D_{R})$. A standard sufficient condition for tightness (see for example, \cite[Lemma 4.2]{KK14}) is  
 $$\sup_{n \in \mathbb{N}, z \in D_{R}} \E |P_n(z)|^2 < \infty.$$ 
 Since
    \begin{equation*}
        \E|P_n(z)|^2 = \E(P_n(z) \overline{P_n(z)}) = \E\sum_{k = 0}^{n} |z|^{2k} \le \sum_{k = 0}^{n} R^{2k} \le \dfrac{1}{1 - R^2} < \infty,
    \end{equation*}
    we obtained the desired tightness. This also completes the proof of Lemma \ref{lem:FunctionalLimitTheorem}.
\end{proof}

\subsection{Convergence in distribution of the number of real roots}\label{sec:distribution}
From the convergence of $P_n$, we now show that $N_{P_{n}}([r, R]) \xrightarrow[]{d} N_{P_{\infty}}([r, R])$. We state the following analogue of \cite[Lemma 4.1]{Iksanov2016}, whose proof carries over almost verbatim. Roughly speaking, the lemma asserts that the mapping $P \mapsto N_{P}[r, R]$ is continuous, assuming that $P_{\infty}$ does not have double roots nor roots at the endpoints.   The result is a consequence of Hurwitz's Theorem (see for example, \cite[Page 152]{conway2012functions}). In essence, Hurwitz’s theorem ensures the continuity of the map that associates to an analytic function the point process formed by its complex zeros. 
When focusing on real zeros, one must additionally verify that these remain confined to the real line under small perturbations. 
This stability holds for analytic functions that are real-valued on $\mathbb{R}$ (as in the case of $P_n$ and $P_{\infty}$): non-real zeros necessarily occur in complex conjugate pairs, and a simple real zero cannot leave the real axis under a sufficiently small perturbation of the function. This explains the requirement that $P_{\infty}$ have neither multiple roots nor zeros at the endpoints.
\begin{lemma}\label{lem:IksanovLemma41}
	Let $A = A[a, b] \subseteq \mathcal H_{\mathbb{R}}(D_R)$ be the set of all $f \in \mathcal H_{\R}(D_R)$ which do not have multiple real zeros in $[a, b]$ and satisfy $f(a),f(b) \neq 0$. Then the set $A$ is open.
	
	Moreover, let $N = N_f([a, b])$ be the number of real zeros of $f$ in $[a, b]$. Then the map $N: \mathcal H_{\R}(D_R) \rightarrow \N$ is locally constant on A (i.e. for every $f \in A$ there is an open neighborhood of $f$ in $\mathcal H_{\R}(D_R)$ on which $N$ is constant), and hence, is continuous.
\end{lemma}

To conclude that $N_{P_{n}}([r, R]) \xrightarrow[]{d} N_{P_{\infty}}([r, R])$, it remains to show that almost surely, $P_\infty$ does not have double roots and $P_{\infty}(r) \neq 0$, $P_{\infty}(R) \neq 0$. 
As mentioned before, the first part follows from  \cite{MichelenYakir2025}.
\begin{lemma}\cite[Theorem 1.3]{MichelenYakir2025}\label{lm:dbr}
	Suppose $\xi$ is such that $\E\log(1 + |\xi|) < \infty$. Then $$\P(\exists z \in B(0, 1): P_{\infty}(z) = P'_{\infty}(z) = 0) = \P(P_{\infty}(0) = P'_{\infty}(0) = 0) = \P(\xi = 0)^2.$$
	
	Consequently, 
	$$\P(\exists x \in [r, R]: P_{\infty}(x) = P'_{\infty}(x) = 0) = 0.$$
\end{lemma}
We note that their proof extends to the case where the $\xi_i$ are not identically distributed, provided they have a finite $(2+\varepsilon_0)$-moment. For the second part, we show that
\begin{lemma}\label{lm:x}
	 For any fixed $x \in \R\setminus \{0\}$, we have
	 $$\P(P_{\infty}(x)=0)=0.$$
\end{lemma}
To prove Lemma~\ref{lm:x}, we will rely on a classical anti-concentration bound. 
In the setting where the $\xi_i$ are i.i.d.  Rademacher variables, this bound appears as 
Erd\H{o}s’s celebrated lemma for the Littlewood--Offord problem~\cite{erdos1945lemma}. 
When the $\xi_i$ are i.i.d.  with a more general distribution, the statement is known as the 
L\'evy--Kolmogorov--Rogozin inequality (see, for instance,~\cite{esseen1968concentration}). 
An extension to independent random variables with bounded 
$(2+\varepsilon_0)$-moment was later obtained in \cite[Lemma~4.1]{DoONguyenVu2018}. 
\begin{lemma}[{\cite[Lemma~4.1]{DoONguyenVu2018}}]\label{lem:ess}
	Let $(\xi_j)$ be independent random variables with variance 1 and are either i.i.d. or have bounded $(2+\ep_0)$-moment.
	Let $c_j \in \mathbb{C}$ 
	be such that $|a_j|\geq a$. 
	There exists a constant $C>0$, depending only on the distribution of $a_0$, such that 
	\[
	\sup_{w \in \mathbb{C}} \, \mathbb{P}\!\left( \left| \sum_{j=1}^{n} a_j \xi_j- w \right| \le a \right) 
	\;\le\; \frac{C}{\sqrt{n}}.
	\]
\end{lemma}

\begin{proof}[Proof of Lemma \ref{lm:x}]
Let $C$ be the constant in Lemma \ref{lem:ess}. For every $n\in \N$, it suffices to show that 
$$\P(P_\infty(x)=0)\le \frac{C}{\sqrt{n}}.$$
By independence of the $\xi_i$, we may separate the contribution of the first $n$ terms from the tail. 
Conditioning on the value of the tail sum $\sum_{j=n}^{\infty} \xi_i x^i$, we obtain
\[
\P\!\big(P_\infty(x)=0\big) 
= \P\!\left(\sum_{j=0}^{n-1} \xi_i x^i +\sum_{j=n}^{\infty} \xi_i x^i=0\right) 
\;\le\; \sup_{w\in \C} \P\!\left(\sum_{j=0}^{n-1} \xi_i x^i - w=0\right).
\]
Applying the anti-concentration bound now gives
\[
\P\!\big(P_\infty(x)=0\big) \;\le\; \frac{C}{\sqrt{n}}
\]
as claimed.
\end{proof}

\subsection{Uniform integrability}\label{sec:mean}
To guarantee convergence in mean, it suffices (for example, see \cite[Theorem 3.5]{billingsley99}) to show that for all $0 < r < R < 1$, $(N_{P_{n}}([r, R])) _{n \ge 1}$ is uniformly integrable. This follows from a more general estimate on moments of $N_n := N_{P_n}([-R, R])$, which we state below.
\begin{lemma}\label{lem:momentsOfNumRoots}
	Denote $N_n := N_{P_n}([-R, R])$. Then for all $\kappa \ge 1$, there exists a constant $M = M(R, \kappa, \varepsilon_0, M_0) \in (0,\infty)$ such that for all $n$, $$ \E N_n^{\kappa} \le M,$$
	where $\varepsilon_0, M_0$ are such that $\E|\xi_k|^{2 + \varepsilon_0} \le M_0$ for all $k$. 
\end{lemma}
From the lemma above, uniform integrability holds: for any $\ep > 0$, we have, for all $n$, 
\begin{align*}
	\E N_{P_{n}}([r, R]) 1_{\{N_{P_{n}}([r, R]) \ge K \}} &\le \left(\E N_{P_{n}}([r, R])^2\right)^{1/2} \P\left( N_{P_{n}}([r, R]) \ge K \right)^{1/2} \\
	&\le \left(\E N_{P_{n}}([r, R])^2\right)^{1/2} \left( \dfrac{\E N_{P_{n}}([r, R])^2}{K^2}\right)^{1/2} \\
	&\le \dfrac{\E N_{P_{n}}([r, R])^2}{K} \le \dfrac{M}{K},
\end{align*}
and we can make the last term at most $\ep$ for $K$ sufficiently large (independent of $n$).

Before proving Lemma \ref{lem:momentsOfNumRoots}, we recall the classical Jensen's inequality which asserts that for every entire function $f$, every $z\in \C$ and $0<r< R$, 
\begin{equation}
	N_{f}\left (B(z, r)\right )\le \frac{\log \frac{\sup_{w\in B(z, R)}|f(w)|}{|f(z)|}}{\log \frac{R}{r}} \label{eq:Jensen}
\end{equation}
This is a corollary of the classical Jensen's formula (see, for example, \cite{Ru}).
\begin{proof}[Proof of Lemma \ref{lem:momentsOfNumRoots}]
For any $d \in (0, 1)$, for all $k$, by H\"{o}lder's inequality,
\begin{equation*}
		1 = \E |\xi_k|^2 = \E |\xi_k|^2 1_{\{|\xi_k| \le d\}} + \E |\xi_k|^2 1_{\{|\xi_k| > d\}} \le d^2 + P(|\xi_k| > d)^{\ep_0/(2 + \ep_0)} M_0^{2/(2 + \ep_0)},
\end{equation*}
so we can choose $d = d(\varepsilon_0, M_0) \in (0, 1)$ sufficiently small such that 
\begin{equation}
	\P(|\xi_k| \le d) < 1 - d =: q \in (0, 1).\label{eq:small}
\end{equation}
 Let 
$$\mathcal B_k := \{|\xi_0|\le d, \dots, |\xi_{k-1}|\le d, |\xi_k|>d\},$$
 then $\P(\mathcal B_k) \le q^k$ for all $k \ge 1$. By Mean Value Theorem and Jensen's inequality \eqref{eq:Jensen}, on $\mathcal B_k$ where $k \le n+1$, we have
\begin{align}
	N_n &\le k + N_{P_n^{(k)}}([-R, R]) \le k + \dfrac{\log \frac{M'}{P^{(k)}(0)}}{\log \frac{\tilde{R}}{R}} \le k + \dfrac{\log S_k - \log d}{\log \frac{\tilde R}{R}}, \label{eq:positivityofLogRk}
\intertext{where $\tilde{R} := (R+1)/2<1$, $M' := \sup_{|z| = \tilde{R}} |P^{(k)}(z)|$, and $S_k := \sum_{j= k}^{n}\binom{j}{k}|\xi_j| \tilde{R}^{j - k}$. (When $k = n+ 1$, the bound is simply $n$) Thus,}
	N_n^{\kappa} &\le 3^{\kappa} \left| \dfrac{\log d}{\log \frac{\tilde R}{R}} \right|^{\kappa} + 3^{\kappa}k^{\kappa}+ 3^{\kappa} \left(\dfrac{|\log S_k|}{\log \frac{\tilde{R}}{R}}\right)^{\kappa}. \nonumber 
\end{align}
Taking expectation on both sides, it remains to show a uniform-in-$n$ bound on 
\begin{equation}\label{eq:tobeBoundMoment}
	\E \left( \sum_{k = 0}^{n} 1_{\mathcal B_k} |\log S_k|^{\kappa}\right) + \E\left( \sum_{k =0}^{n} 1_{\mathcal B_k} k^{\kappa} \right). 
\end{equation} 
The second summation is easy since $\sum_{k = 0}^{\infty} q^k k^{\kappa} < \infty$ for fixed $\kappa$. For the first summation, note that $\E |\xi_j|\le (\E \xi_j^2)^{1/2}=1$, thus, 
\begin{equation*}
 \E S_k \le  \sum_{j \ge 0} \binom{k + j}{k} \tilde{R}^j =  (1 - \tilde{R})^{- (k + 1)}.
\end{equation*}
On $\mathcal B_k$, 
$$ S_k \ge |\xi_k|>d. $$ 
This observation prevents $|\log S_k|$ from blowing up when $S_k$ is small. We partition $\mathcal B_k$ into the events $$ \mathcal B_{ki} := \mathcal B_k \cap  \{e^i \E S_k < S_k \le e^{i +1} \E S_k\}. $$ 
Moreover, denote our cut-off $i_k := \lfloor q^{-k/(2\kappa)} \rfloor \ge 1$. Then 
\begin{itemize}
	\item On $\bigcup_{i < i_k} \mathcal B_{ki}$, we have 
	$$\log d\le \log S_k \le i_k (k + 1) \log (1 - \tilde R)^{-1}.$$
	Therefore,
	\begin{align}
		\E\left(\sum_{i < i_k} 1_{\mathcal B_{ki}} |\log S_k|^{\kappa}\right) &=O_\kappa(1) q^k \left( \left| \log d\right|^{\kappa} + i_k^{\kappa}(k + 1)^{\kappa} \left( \log(1 - \tilde R)^{-1} \right)^{\kappa} \right) \nonumber \\
		&=O_\kappa(1) q^{k/2} k^{\kappa}. \label{eq:BulkMomentOne}
	\end{align}
	
	\item For $i \ge i_k > 0$, by Markov's inequality, $\P(\mathcal B_{ki} \mid \mathcal B_k) \le e^{-i}$. Moreover, on $\mathcal B_{ki}$, $$ 0 \le i \log \E S_k \le \log S_k \le (i + 1) (k + 1) \log (1 - \tilde R)^{-1}.  $$ Thus, 
	\begin{align}
		\E\left(\sum_{i \ge i_k} 1_{\mathcal B_{ki}} |\log S_k|^{\kappa}\right) &\le q^k \sum_{i \ge i_k} e^{-i} (i + 1)^{\kappa} (k + 1)^{\kappa}   \left( \log(1 - \tilde R)^{-1} \right)^{\kappa} \nonumber \\
		&= O_\kappa(1) q^{k/2} k^{\kappa} . \label{eq:BulkMomentTwo}
	\end{align}
\end{itemize} 
Combining \eqref{eq:BulkMomentOne}, \eqref{eq:BulkMomentTwo}, and summing over all $k$ gives the desired uniform-in-$n$ bound for the first summation in \eqref{eq:tobeBoundMoment}, completing our proof.
\end{proof}

\section{Proof of \eqref{eq:near0}: Roots near the origin}\label{sec:ConvNear0}
In this section, we prove \eqref{eq:near0} which we restate as follows.
\begin{lemma}\label{lm:ConvNear0}
	For all $\ep > 0$, there exists $\delta, N_0 > 0$ such that for all $n \ge N_0$,
	\begin{equation*}
		\E N_{P_n}((0, \delta)) < \ep.
	\end{equation*}
\end{lemma}

\begin{proof} 
In the i.i.d. setting of Theorem \ref{thm:mainthm1}, since the intersection of the decreasing sequence of events $(\{\xi|<1/k, \xi\neq 0\})_{k\in \N}$ has empty intersection, it holds that
	$$\lim_{k\to\infty} \P( |\xi| < 1/k, \xi \neq 0)=0.$$
Thus, for any $\ep \in (0, 1)$, there exists $\delta\in (0, \ep)$ such that 
	\begin{equation}
		\P( |\xi| < \delta^{1/2}, \xi \neq 0) < \ep.\label{eq:xi:small}
	\end{equation}
		This together with \eqref{cond:decay0} allow to deduce the same bound for all $\xi_k$ in Theorem \ref{thm:mainthm}.
	
	Moreover, let $\tau := \min\{k: \xi_k \neq 0\}$ and denote $P_n^{\tau}(x) := \sum_{k = \tau}^{n} \xi_{k} x^{k - \tau}$. If $\tau>n$, $P_n$ has all roots zero, and hence, $ N_{P_n}((0, \delta)) =0$. Hence, for the rest of the proof, we only consider the event that $\tau\le n$. 
By Jensen's inequality \eqref{eq:Jensen} applied to the balls $B(0, \delta)$ and $B(0, 2\delta)$, the above is at most
		\begin{align*}	
			N_{P_n}((0, \delta)) =N_{P^{\tau}_n}((0, \delta)) &\le\dfrac{1}{\log 2} \log \dfrac{|\xi_{\tau}| + |\xi_{\tau+1}| \delta + \dots + |\xi_n| \delta^{n - \tau}}{|\xi_{\tau}|}.
	\end{align*}
	From the above, together with the inequality $\log (1 + x) \le x$, on the event that $|\xi_{\tau}| > \delta^{1/2}$,
	\begin{align}\label{eq:near0TauSmallDeltaGreat}
		\E 1_{\{|\xi_{\tau}| > \delta^{1/2}\}}N_{P_n}((0, \delta)) \le \dfrac{1}{\log 2} \E \left( 1_{\{|\xi_{\tau}| > \delta^{1/2}\}} \dfrac{\sum_{k = \tau+ 1}^{n} |\xi_k| \delta^{k - \tau}}{\delta^{1/2}} \right) \le  \dfrac{\delta^{1/2}}{(1 - \delta)\log 2}= O(\ep).
	\end{align}
\vspace{5mm}
On $\{|\xi_{\tau} | \le \delta^{1/2}\}$, by the interlacing of roots,  
	\begin{align}
		N_{P_n^{\tau}}(0, \delta) &\le 1 + N_{(P_n^{\tau})' }((0, \delta)), \nonumber
\end{align}
therefore,
	\begin{align}		\E  1_{\{|\xi_{\tau}| \le \delta^{1/2}\}}N_{P_n}((0, \delta)) &\le \E 1_{\{|\xi_{\tau}| \le \delta^{1/2}\}}  (1 + N_{(P_n^{\tau})' }((0, \delta))) 
			\end{align}
		 
and note that $\xi_{\tau}$ is independent of $(P_n^{\tau})'$, so by our choice of $\delta$, the above is at most
	\begin{align}
		& \ep (1 + \E N_{(P_n^{\tau})' }((0, \delta))) =  O(\ep), \label{eq:near0TauSmallDeltaSmall}
	\end{align}
where	the last inequality holds due to Lemma \ref{lem:momentsOfNumRoots} (the lemma is stated for $P_n$ but the proof can be adapted to $(P_n^{\tau})'$). Equations \eqref{eq:near0TauSmallDeltaGreat} and \eqref{eq:near0TauSmallDeltaSmall} together completes the proof for Lemma \ref{lm:ConvNear0}.
\end{proof}

 \section{Proof of Corollary \ref{cor:C}} \label{sec:cor}
 We recall from Theorem \ref{thm:universality} that for any $\ep>0$, there exist constants $C_0, N_0$ depending only on $\ep, \ep_0$ and $M_0$ such that for all $n\ge N_0$ and $C\ge C_0$,
 \begin{equation}
 		\left |	 \E N_{P_n}([1-C^{-1}, 1]) - \left (\frac{1}{2\pi}\log\frac{n}{C}+ C_{\mathrm{Gau}, [0, 1]}- \frac{1}{2\pi}\log 2\right ) \right |=:|R_{n, C, \xi_0, \xi_1,\dots}|\le \ep. \label{eq:diff:ep}
 	\end{equation}
  	
 	We have also proved that for any constant $\delta>0$,
 	\begin{equation}
 	\lim_{n\to\infty}	\E N_{P_n}((1-\delta, 1-C^{-1})) = \E N_{P_\infty}((1-\delta, 1-C^{-1})).
 	\end{equation}
 	Moreover, for $\delta$ sufficiently small, by \eqref{eq:near0},
 	\begin{equation}
 		\E N_{P_n}((0, 1-\delta])   \le \ep,\notag
 	\end{equation}
which also holds for $P_{\infty}$, using the same proof.
 	Therefore, for any $C>0$,
 	$$\lim_{n\to\infty}\E N_{P_n}((0, 1-C^{-1})) = \E N_{P_\infty}((0, 1-C^{-1})).$$
 	All in all, 
%
for all $C\ge C_0$ as defined in \eqref{eq:diff:ep}, 
\begin{eqnarray*}
	C_{\xi_0, \xi_1, \dots, (0, 1]} &=& \lim_{n\to\infty} \left [\E N_{P_{n}}((0, 1]) -\frac{1}{2\pi}\log n\right ]\\
	&=& \lim_{n\to\infty}  \left [\E N_{P_{n}}((0, 1-C^{-1}))  +\left (-\frac{1}{2\pi}\log C+ C_{\mathrm{Gau}, [0, 1]}- \frac{1}{2\pi}\log 2\right )  + R_{n, C, \xi_0, \xi_1,\dots}\right ]\\
	&&\quad\quad\quad\quad \text{ where $\limsup_{n}|R_{n, C, \xi_0, \dots}|\le \ep$ by \eqref{eq:diff:ep}}\\
	&=& \E N_{P_{\infty}} \left (\left (0, 1-C^{-1}\right )\right )  -\frac{1}{2\pi}\log C+ C_{\mathrm{Gau}, [0, 1]}- \frac{1}{2\pi}\log 2  + \lim_{n\to\infty}  \left [R_{n, C, \xi_0,\xi_1, \dots}\right ].
\end{eqnarray*}
Thus, 
\begin{eqnarray*}
	&&\left |	C_{\xi_0, \xi_1, \dots, (0, 1]}  - \left (\E N_{P_{\infty}} \left (\left (0, 1-C^{-1}\right )\right )  -\frac{1}{2\pi}\log C+ C_{\mathrm{Gau}, [0, 1]}- \frac{1}{2\pi}\log 2 \right ) \right |\\
	&&\le \limsup_{n\to\infty}  \left |R_{n, C, \xi_0,\xi_1, \dots}\right |\le \ep
\end{eqnarray*}
proving uniform convergence.

\section{Proof of Theorem \ref{thm:cont}}\label{sec:cont}

We start by constructing a coupling of the random series $Q_{n}=\sum_{k=0}^{\infty} \xi_{k, n} x^{k}$, with coefficients $\xi_{k,n}\sim \mu_{n}$, and $Q_\infty=\sum_{k=0}^{\infty} \xi_{k,\infty} x^{k}$, with coefficients $\xi_{k, \infty}\sim \mu_\infty$, such that $Q_{n}$ converges to $Q_{\infty}$ almost surely.

\begin{lemma}\label{lem:SkorokhodConstruction}
	Assume that the sequence of probability measures $\mu_{n}$ converges to $\mu_\infty$ in distribution, then we can construct $(\xi_{k, n})_{k, n\in \N}$ and $(\xi_{k, \infty})_{k\in \N}$ such that the following holds.
	\begin{enumerate}
		\item For each $n\in \N\cup \{\infty\}$, the random variables $\xi_{0, n}, \xi_{1, n}, \dots$ are i.i.d. with distribution $\mu_{n}$;
		\item For each $k$, $\xi_{k, n}$ converges to $\xi_k$ almost surely as $n\to\infty$.
	\end{enumerate}
\end{lemma}
\begin{proof}
	Let $U_0, U_1, U_2, \dots$ be i.i.d. uniformly distributed in $[0, 1]$. Moreover, let $F_n$ and $F_\infty$ be cumulative distribution functions (cdfs) of $\mu_{n}$ and $\mu_\infty$ respectively. For $a \in (0, 1)$, the pseudo-inverse (or quantile function) of $F_\infty$ is defined by
	$$ F_\infty^{-1}(a) := \inf\{ x \in \R: F_\infty(x) \ge a \},$$
	and $F_n^{-1}$ similarly. We set
	$$ \xi_{k, n} := F_n^{-1}(U_k), \qquad \xi_{k, \infty} := F_\infty^{-1}(U_k) .$$
It is not hard to see that these random variables satisfy the requirements. Indeed, the independence across the first index is given by the independence of $U_0, U_1, \dots$. As for the marginal distribution, for all $x\in \R$, we have
$$\P(\xi_{k, \infty}\le x) = \P(F_\infty^{-1}(U_k)\le x)=\P(F_\infty(x)\ge U_k) = F_\infty(x).$$
Therefore, $\xi_{k, \infty}\sim \mu_\infty$. Similarly, $\xi_{k, n}\sim \mu_{n}$.

Next, by the assumption $F_n \overset{d}{\to} F_\infty$, for each continuity point $x$ of $F_\infty$ we have $F_n(x)\to F_\infty(x)$. 
It follows that for each continuity point $u$ of $F_\infty^{-1}$,
\[
F_n^{-1}(u) \;\longrightarrow\; F_\infty^{-1}(u).
\]
Since $F_\infty^{-1}$ is monotone, the set of its discontinuities is at most countable. 
Therefore
\[
\mathbb{P}\!\left(U_k \text{ lies at a discontinuity of } F_\infty^{-1}\right) = 0.
\]
On the event that $U_k$ is not a discontinuity, we obtain $\xi_{k, n} = F_n^{-1}(U_k)\to F_\infty^{-1}(U_k)=\xi_{k,\infty}$. 
Hence $\xi_{k, n}\to \xi_{k, \infty}$ almost surely, completing the proof.
\end{proof}

For the rest of this proof, we will use this construction and the corresponding $Q_{n}$. In order to prove Theorem \ref{thm:cont}, we only need to prove continuity of the map $\mu\to C_{\mu, (0, 1]}$ which is the constant associated with the interval $(0, 1]$. Indeed, by taking reciprocal as before, we get for any probability measure $\mu$,
$$C_{\mu, (0, 1]} = C_{\mu, [1, \infty)} = C_{\mu, (1, \infty)}$$
where the last equality follows due to Lemma \ref{lm:x} for $x = 1$. Our proof works verbatim to show that $C_{\mu, [-1, 0)} = C_{\mu, (-\infty, -1)} $ is continuous. Finally, as $\mu_n(0)\to\mu_{\infty}(0)$, the contribution from the multiplicity at $0$ is also continuous. Summing up, we obtain continuity of $C_{\mu}$ as
$$C_{\mu}=C_{\mu, (-\infty, -1)} +C_{\mu, [-1, 0)} + C_{\mu, (0, 1]} + C_{\mu, (1, \infty)}+\E N_{P_{\mu}}(\{0\})$$
where $P_{\mu} = \sum_{k=0}^{\infty} \xi_k x^{k}, \ \xi_k\overset{i.i.d.}{\sim}\mu.$

 To prove continuity of $\mu\to C_{\mu, (0, 1]}$, recall by Corollary \ref{cor:C} that for every $\ep>0$, there exists a constant $C$ such that for all $n\in \N\cup \{\infty\}$,
$$\left |	C_{\mu_n, (0, 1]}  - \left (\E N_{Q_{n}} \left (\left (0, 1-C^{-1}\right )\right )  -\frac{1}{2\pi}\log C+ C_{\mathrm{Gau}, [0, 1]}- \frac{1}{2\pi}\log 2 \right ) \right |\le \ep.$$

All that is left to prove is, for a fixed large constant $C$, to show that 
$$|\E N_{Q_{n}} \left (\left (0, 1-C^{-1}\right )\right ) - \E N_{Q_{\infty}} \left (\left (0, 1-C^{-1}\right )\right ) |\le \ep \text{ as } n\to\infty.$$ 
To this end, we continue to split the interval into $\left (\delta, 1-C^{-1}\right )$ and $(0, \delta)$ as before. The result is a combination of Lemma \ref{lem:ConvFarFrom01Cont} and Lemma \ref{lm:small} stated below. 
The first lemma is an analogue of Lemma \ref{lem:ConvFarFrom01}.
\begin{lemma}\label{lem:ConvFarFrom01Cont}
For all $0 < r < R < 1$, we have $$\lim_{n \rightarrow \infty}\E N_{Q_{n}}[r, R] = \E N_{Q_\infty}[r, R].$$
\end{lemma}

\begin{proof}
	Firstly, we show that $Q_n$ convergence to $Q_\infty$ in finite-dimensional distributions. In particular, we show that for all $l \in \mathbb{N}$ and $z_1, \dots, z_l \in D_{R}$, as $n \rightarrow \infty$,
	\begin{equation*}
		\begin{pmatrix}
			\text{Re } Q_n(z_1) \\
			\text{Im }  Q_n(z_1) \\
			\vdots \\
			\text{Re } Q_n(z_l) \\
			\text{Im } Q_n(z_l)
		\end{pmatrix} \xrightarrow[]{d} \begin{pmatrix}
			\text{Re } Q_\infty(z_1) \\
			\text{Im } Q_\infty(z_1) \\
			\vdots \\
			\text{Re } Q_\infty(z_l) \\
			\text{Im } Q_\infty(z_l)
		\end{pmatrix}.
	\end{equation*}
	To this end, we construct the coefficients of $Q_n$ and $Q_\infty$ according to Lemma \ref{lem:SkorokhodConstruction}. With this construction, it suffices to show that for all $z \in D_R$, a.s. $Q_n(z) \to Q_\infty(z)$.
	
	Consider the event $\mathcal  A = \left\{\sum_{k \ge 0} \left(\sup_n |\xi_{k, n}|\right) z^k < \infty \right\}$. On $\mathcal A$, we have $Q_n (z) \to Q_\infty(z)$ by Dominated Convergence Theorem, so it remains to show that $\P(\mathcal A) = 1$. If $U_k$ is uniformly distributed on $[0, 1]$, then for all $\varepsilon > 0$, 
	\begin{align*}
		\P\left( \sup_n |\xi_{k, n}| \ge (1 + \varepsilon)^k  \right) &= \P\left( \exists n: |\xi_{k, n}| \ge (1 + \varepsilon)^k \right) \\
		&\le\P\left( \exists n: \xi_{k, n} \ge (1 + \varepsilon)^k \right) + \P\left( \exists n: \xi_{k, n} \le - (1 + \varepsilon)^k \right).
\end{align*}
If $n$ is such that $\xi_{k, n} := F_n^{-1}(U_k) \ge (1 + \varepsilon)^k$, then $U_k \ge F_n((1 + \varepsilon)^k) \ge 1 - (1 + \varepsilon)^{-2k}$, where the last equality follows from Chebyshev's inequality and the fact that $F_n$ has unit variance. Thus,
\begin{align*}
		\P\left( \exists n: \xi_{k, n} \ge (1 + \varepsilon)^k \right) &\le \P\left( U_k \ge  1 - (1 + \varepsilon)^{-2k}  \right) \le (1 + \varepsilon)^{-2k}. 
\end{align*}
We obtain the same bound for $\P\left( \exists n: \xi_{k, n} \le - (1 + \varepsilon)^k \right)$. Thus,
\begin{align*}
		\P\left( \sup_n |\xi_{k, n}| \ge (1 + \varepsilon)^k  \right) &\le 2( 1+ \varepsilon)^{-2k},  \\ 
		\sum_{k \ge 0 }\P\left( \sup_n |\xi_{k, n}| \ge (1 + \varepsilon)^k  \right) &\le 2 \sum_{k \ge 0}( 1+ \varepsilon)^{-2k} < \infty.
	\end{align*}
	By Borel-Cantelli lemma, a.s. there are finitely many $k$ such that $\sup_n |\xi_{k, n}| \ge (1 + \varepsilon)^k$. Pick $\varepsilon > 0$ sufficiently small such that $R (1 + \varepsilon) < 1$, then a.s. $\mathcal A$ happens.
	
	The remaining steps are identical to the proof of Lemma \ref{lem:ConvFarFrom01} as tightness is easily verified, and uniform integrability follows directly since all $\mu_n, \mu_\infty$ share the same parameters $\varepsilon_0, M_0$ in Lemma \ref{lem:momentsOfNumRoots}.
\end{proof}

For the region near $0$, we have the analogue of Lemma \ref{lm:ConvNear0}.
\begin{lemma}\label{lm:small}
	In the same setting as Lemma \ref{lem:ConvFarFrom01Cont}, for all $\varepsilon > 0$, there exists $\delta, N_0$ (possibly depending also on $\mu$) such that for all $n \ge N_0$, $$ 0\le \E N_{Q_n}((0, \delta)) < \varepsilon.$$
\end{lemma}

\begin{proof}
	Tracing the proof of Lemma \ref{lm:ConvNear0}, it suffices to show that for all $\varepsilon > 0$, there exists $\delta, N_0$ such that for all $n \ge N_0$, with $\xi_n \sim \mu_n$,
	$$ \P\left( |\xi_n| \le \delta, \xi_n \neq 0 \right) < \varepsilon.$$ 
Since $\mu_\infty(0)<1$, there exists $\delta_0>0$ such that for $\xi_{\infty}\sim \mu_\infty$ and for all $\delta<\delta_0$, 
$$ \P\left( |\xi_\infty| \le \delta, \xi_\infty \neq 0 \right) \le  \P\left( |\xi_\infty| \le \delta_0, \xi_\infty \neq 0 \right) < \varepsilon/2.$$
Let $\delta<\delta_0$ be such that $\mu_\infty(\pm \delta) = 0$.  By convergence of $\mu_n$ to $\mu_\infty$, it holds that
$$\P(-\delta<\xi_n<\delta)\to \P(-\delta<\xi_n<\delta).$$
Since $\mu_n(0)\to \mu_\infty(0)$, this yields
$$ \P\left( |\xi_n| \le \delta, \xi_n \neq 0 \right) \to \P\left( |\xi_\infty| \le \delta, \xi_\infty \neq 0 \right)  < \varepsilon/2.$$
Therefore, for sufficiently large $n$, 
 $$ \P\left( |\xi_n| \le \delta, \xi_n \neq 0 \right) \le \ep$$
 as desired.
\end{proof}

\section{Proof of Lemma \ref{lm:nonUniversality}}\label{prof:nonuni}
 
Let $\xi$ be a random variable with positive mass at 0, which we denote by $p$.
Let $\ep = 1/2$ (say). From \eqref{eq:uni}, there exists a large constant $C$ such that for all sufficiently large $n$, we know that the contribution near 1 captures the leading order term, more specifically, 
\begin{equation}
 \E N_{P_n}([1-C^{-1}, 1]) \ge  \frac{1}{2\pi}\log\frac{n}{C}+ C_{\mathrm{Gau}, [0, 1]}- \frac{1}{2\pi}\log 2 -\ep. \notag
\end{equation}
From \eqref{eq:0}, we get the contribution from $0$ is significant
\begin{equation} 
	\E N_{P_n}(0) \ge \dfrac{p}{1 - p} -\ep.\notag
\end{equation}
Altogether, we have
$$ \E N_{P_n}([0, 1]) \ge  \dfrac{\log n}{2\pi} + \dfrac{p}{1 - p} -\dfrac{\log C}{2\pi} + C_{\mathrm{Gau}, [0, 1]}- \frac{1}{2\pi}\log 2-2\ep.$$
As $p$ gets arbitrarily close to $1$, the term $\dfrac{p}{1 - p}$ dominates the other constants, showing that the final constant term $C_{\xi, [0, 1]}$ can be arbitrarily large. Using the same argument as in beginning of the main proof, this can be extended to the whole real line, completing the proof.
 
\section{An alternative argument for symmetric distribution on $\{-1, 0, 1\}$}\label{sec:shortTaylor}
Of independent interest, we provide a short, alternative argument for showing that the probability that $P_{\infty}$ has double root is 0 (as stated in Lemma \ref{lm:dbr}) in the special case where $\xi$ has symmetric distribution in $\{-1, 0, 1\}$ with mass $\P(\xi = 0) = q>0$. In particular, we show the following.
\begin{theorem}
Almost surely, $P_{\infty}$ does not have a double root in $(0, 1)$.
\end{theorem}
\begin{proof}
	Let $R\in (0, 1)$. It suffices to prove the statement for $(0, R)$.
Denote $$\mathcal A := \left\{{\bf \xi} = (\xi_0, \xi_1 \dots) \in \{-1, 0, 1\}^{\N}: \sum_{k \ge 0} \xi_k x^k \text{ has a double root in } (0, R)\right\}.$$
We will prove that $\P(\mathcal A) =0$ via L\'{e}vy's zero-one law. First,
	\begin{align*}
		\P({\bf \xi} \in \mathcal A, \xi_0=0) &= q\P\left(\sum_{k \ge 1} \xi_k z^k \text{ has a double root in } (0, R) \right)= q\P(\mathcal A)
	\end{align*}
	since $(\xi_1, \xi_2, \dots)$ has the same distribution as $(\xi_0, \xi_1, \dots)$.
	By symmetry, $\P(\mathcal A_1\vert \xi_0=1) = \P(\mathcal A_{-1}\vert \xi_0=1)$. 
	Since $ \mathcal A_{-1}, \mathcal A_0, \mathcal A_1$ forms a partition of $\mathcal A$, it follows that 
	$$\P(\xi\in \mathcal A\vert \xi_0=\pm 1) =  \frac{1}{1-q} \P(\xi\in \mathcal A, \xi_0=\pm 1) = \frac{1-q}{1-q} \P(\xi\in \mathcal A)=\P(\mathcal A).$$
	Therefore,conditioning on $\xi_0$ does not change the probability of $P_{\infty}$ getting a double root in $(0, R)$. Extending the above, conditioning on $\xi_0, \dots, \xi_n$ does not change the probability of $\mathcal A$.
	
	Denote $\mathcal F_n := \sigma\left( \{\xi_k\}_{k = 0}^n \right)$, then $\mathcal F_0 \subseteq \mathcal F_1 \subseteq \dots \subseteq \mathcal F_{\infty}$. We then have, as $n\to\infty$,
	$$ \E 1_{\mathcal A} = \E (1_{\mathcal A} \mid \mathcal F_n) \xrightarrow[]{a.s.} \E(1_{\mathcal A} \mid \mathcal F_{\infty}) = 1_{\mathcal A},$$
	the first equality is due to our observation of conditioning on the first $n$ coefficients, the convergence is due to L\'{e}vy's zero-one law, and the last equality holds since $\mathcal A$ is $\mathcal F_{\infty}$-measurable. Thus, $\P(\mathcal A) \in \{0, 1\}$. To see that this probability cannot be $1$, note that there exists $N$ sufficiently large such that when all the first $N$ entries are 1, namely $\xi_0 = \xi_1 = \dots = \xi_N = 1$, then for all $x \in (0, R)$, 
	\begin{align*}
	|P_\infty(x)|\ge	\sum_{k = 0}^N x^k - \left| \sum_{k > N} \xi_kx^k \right| &\ge \dfrac{1 - x^{N+1}}{1 - x} - \sum_{k >N} x^k = \dfrac{1 - 2x^{N+1}}{1 - x} > 0.
	\end{align*}
	Thus, with positive probability $P_{\infty}$ does \textit{not} have a root in $(0, R)$, i.e. $\P(\mathcal A) < 1$. This completes our proof.
\end{proof}

\bibliographystyle{alpha}
\bibliography{references}

\end{document}